\documentclass[12pt,a4paper]{amsart}
\usepackage{amsmath,amssymb,amsthm,latexsym,MnSymbol}
\usepackage[utf8]{inputenc}
\usepackage{url}
\usepackage[pdftex]{graphicx}
\usepackage{tikz}
\usepackage[pdftex]{hyperref}
\usepackage[all]{xy}

\newcommand{\Z}{\mathbf{Z}}
\newcommand{\Q}{\mathbf{Q}}

\newcommand{\R}{\mathbf{R}}
\newcommand{\C}{\mathbf{C}}
\newcommand{\PP}{\mathbf{P}}
\newcommand{\A}{\mathbf{A}}

\DeclareMathOperator{\Aut}{Aut}

\DeclareMathOperator{\GL}{GL}

\DeclareMathOperator{\cInd}{c-Ind}
\DeclareMathOperator{\id}{id}

\DeclareMathOperator{\ord}{ord}

\DeclareMathOperator{\SL}{SL}

\DeclareMathOperator{\tr}{tr}

\newtheorem{thm}{Theorem}[section]
\newtheorem{lem}[thm]{Lemma}
\newtheorem{pro}[thm]{Proposition}

\theoremstyle{definition}
\newtheorem{definition}[thm]{Definition}

\newtheorem{pb}[thm]{Problem}

\theoremstyle{remark}
\newtheorem{remark}[thm]{Remark}

\begin{document}

\title[Ramification at the cusps]{On the ramification of modular parametrizations at the cusps}

\author{François Brunault}
\address{François Brunault\\
ÉNS Lyon\\
Unité de mathématiques pures et appliquées\\
46 allée d'Italie\\
69007 Lyon\\
France}
\email{francois.brunault@ens-lyon.fr}
\urladdr{http://perso.ens-lyon.fr/francois.brunault/}




\begin{abstract}
We investigate the ramification of modular parame\-trizations of elliptic curves over $\Q$ at the cusps. We prove that if the modular form associated to the elliptic curve has minimal level among its twists by Dirichlet characters, then the modular parametrization is unramified at the cusps. The proof uses Bushnell's formula for the Godement-Jacquet local constant of a cuspidal automorphic representation of $\GL(2)$. We also report on numerical computations indicating that in general, the ramification index at a cusp seems to be a divisor of 24.\\

\noindent \textsc{Résumé}. Nous étudions la ramification aux pointes des paramé\-trisations modulaires des courbes elliptiques sur $\Q$. Nous montrons que si le forme modulaire associée à la courbe elliptique est de niveau minimal parmi ses tordues par les caractères de Dirichlet, alors la paramétrisation modulaire est non ramifiée aux pointes. La preuve utilise la formule de Bushnell pour la constante locale de Godement-Jacquet d'une représentation automorphe supercuspidale de $\GL(2)$. Nous présentons également des calculs numériques indiquant qu'en général, l'indice de ramification en une pointe semble être un diviseur de 24.
\end{abstract}

\maketitle

\section{Introduction}

Let $E/\Q$ be an elliptic curve of conductor $N$. It is known \cite{BCDT} that $E$ admits a \emph{modular parametrization}, in other words a non-constant morphism $\varphi : X_0(N) \to E$ defined over $\Q$. By the Riemann-Hurwitz formula, the morphism $\varphi$ necessarily ramifies as soon as the genus of $X_0(N)$ is at least $2$, and we may ask whether its ramification points have interesting properties. In this direction, Mazur and Swinnerton-Dyer discovered a link between the analytic rank of $E$ and the number of ramification points of $\varphi$ on the imaginary axis \cite{mazur-swinnerton-dyer:LpE}. Further results and numerical examples were obtained by Delaunay \cite{delaunay}.

In this article, we consider the following problem.

\begin{pb}\label{pb} Compute the ramification index $e_{\varphi}(x)$ of $\varphi$ at a given cusp $x \in X_0(N)$.
\end{pb}

Let $f_E$ be the newform of weight $2$ on $\Gamma_0(N)$ associated to $E$. The pull-back $\varphi^* \omega_E$ of a Néron differential $\omega_E$ on $E$ under $\varphi$ is a nonzero multiple of $\omega_{f_E} = 2\pi i f_E(z) dz$. It follows that the ramification index of $\varphi$ at a given point $x \in X_0(N)$ is given by $e_{\varphi}(x) = 1+\ord_{x} \omega_{f_E}$. We prove the following result.

\begin{thm}\label{main result}
Let $E/\Q$ be an elliptic curve such that the newform $f_E$ has minimal level among its twists by Dirichlet characters. Then $\omega_{f_E}$ doesn't vanish at the cusps of $X_0(N)$. In particular, the modular parametrization of $E$ is unramified at the cusps of $X_0(N)$.
\end{thm}

\begin{remark}\label{rk Gamma0N}
A newform having minimal level among its twists by Dirichlet characters is said to be \emph{minimal by twist}. It is not true in general that if a newform $f$ of weight $2$ on $\Gamma_0(N)$ is minimal by twist, then $\omega_f$ doesn't vanish at the cusps of $\Gamma_0(N)$. For example, there is a newform $f$ of weight 2 on $\Gamma_0(625)$ such that $f$ is minimal by twist and $\omega_f$ vanishes at the cusp $1/25$ (see Remark \ref{rk 625}).
\end{remark}

If $N$ is squarefree, then all newforms of level $N$ are minimal by twist, and in this particular case, Theorem \ref{main result} follows easily by considering the action of Atkin-Lehner involutions. Thus modular parametrizations of semistable elliptic curves are always unramified at the cusps.

For general $N$, determining the ramification index becomes more intricate and we proceed as follows. In \S\ref{merel formula} we apply a formula of Merel which expresses the translate of a newform $f$ as a linear combination of twists of $f$ by Dirichlet characters. This enables us in \S\ref{reduction local} to reduce Theorem \ref{main result} to a purely local non-vanishing statement. We prove this non-vanishing in \S\ref{sec unramified}-\ref{sec ramified} using Bushnell's formula for the local constant of a cuspidal automorphic representation of $\GL(2)$, together with results of Loeffler and Weinstein on the cuspidal inducing data underlying such representations.

The following side result may be of independent interest. Given a newform $f$ which is minimal by twist, we obtain a rather explicit expression, depending only on the local components of $f$, for the Fourier expansion of $f$ at an arbitrary cusp (see Remark \ref{remark bdn}).

Theorem \ref{main result} was suggested by numerical computations, which we report in \S\ref{numerical}. Using Pari/GP \cite{pari250}, we estimated numerically the ramification indices at the cusps, for all elliptic curves of conductor $\leq 2000$. This provided us with a list of $745$ elliptic curves (up to isogeny) whose modular parametrization seemed to have at least one ramification point among the cusps. Using Magma \cite{magma}, we then checked that none of the corresponding modular forms was minimal by twist. In our examples, the ramification index always appears to be a divisor of 24. It seems interesting to find a general formula for this number in terms of $f$.

After I finished this work, I learned of the recent article \cite{nelson-pitale-saha}, which provides a nice exact formula for an averaged version of the $n$-th Fourier coefficient with respect to cusps of a given level \cite[Prop 3.12]{nelson-pitale-saha}. This can be used to give an explicit formula for the ramification index of a modular parametrization at a cusp (\cite[Remark 3.15]{nelson-pitale-saha},\cite{saha:unpublished}).

In the first version of this work, there was a mistake in the argument of Section \ref{sec unramified}. I would like to thank the anonymous referee for pointing out this mistake, which led me to find the example alluded to in Remark \ref{rk Gamma0N}. I would also like to thank Christophe Delaunay for helpful comments and Hao Chen for confirming numerically the example of Remark \ref{rk 625}. Finally, I would like to thank the anonymous referee for valuable improvement suggestions on this work.

\section{First properties of the ramification index}

In this section, we establish basic properties of the ramification index.

\begin{definition}
Let $f$ be a newform of weight $2$ on $\Gamma_0(N)$, and let $\omega_f = 2\pi if(z)dz$ be the $1$-form associated to $f$. For any point $x \in X_0(N)(\C)$, we define the ramification index of $f$ at $x$ by $e_f(x)=1+\ord_{x} (\omega_f)$.
\end{definition}

\begin{lem}\label{lem WQ}
Let $Q$ be a divisor of $N$ such that $(Q,\frac{N}{Q})=1$, and let $W_Q$ be the corresponding Atkin-Lehner involution on $X_0(N)$. For every $x \in X_0(N)(\C)$, we have $e_f(W_Q (x))=e_f(x)$.
\end{lem}

\begin{proof}
We have $\ord_{W_Q (x)} (\omega_f) = \ord_x (W_Q^* \omega_f) = \ord_x (\omega_f)$ since $f$ is an eigenvector of $W_Q$.
\end{proof}

\begin{lem}\label{lem AutC}
Let $\sigma \in \Aut(\C)$ and let $f^\sigma \in S_2(\Gamma_0(N))$ be the newform obtained by applying $\sigma$ to the coefficients of $f$. For every $x \in X_0(N)(\C)$, we have $e_f(x)=e_{f^\sigma}(\sigma(x))$.
\end{lem}

\begin{proof}
This follows as in Lemma \ref{lem WQ} from $\sigma_* \omega_f = \omega_{f^\sigma}$.
\end{proof}

Recall that the set of cusps of $X_0(N)(\C)$ is $\Gamma_0(N) \backslash \PP^1(\Q)$.

\begin{definition}
The \emph{level} of a cusp $x$ of $X_0(N)$ is defined to be $(b,N)$, where $\frac{a}{b} \in \PP^1(\Q)$ is any representative of $x$ such that $(a,b,N)=1$.
\end{definition}

\begin{lem}\label{lem AutC cusps}
For any divisor $d$ of $N$, the group $\Aut(\C)$ acts transitively on the set of cusps of level $d$ of $X_0(N)$.\end{lem}

\begin{proof}
This is a consequence of \cite[Thm 1.3.1]{stevens:arithmetic}.
\end{proof}

The action of $W_Q$ on the cusps can be described as follows.

\begin{lem}\label{lem WQ cusps}
Let $N=QQ'$ with $(Q,Q')=1$. Let $d$ be a divisor of $N$. Write $d=d_Q d_{Q'}$ with $d_Q | Q$ and $d_{Q'} | Q'$. Then $W_Q$ maps cusps of level $d$ to cusps of level $\frac{Q}{d_Q} \cdot d_{Q'}$.
\end{lem}

\begin{proof}
Since $W_Q$ is defined over $\Q$, it suffices to compute the level of the cusp $W_Q(\frac{1}{d})$. Let $u,v$ be two integers such that $Qu-Q'v=1$. Then $W_Q(\frac{1}{d})=\begin{pmatrix} Qu & v \\ N & Q \end{pmatrix}(\frac{1}{d}) = \frac{Qu+dv}{N+dQ} = \frac{a}{b}$ with $a=\frac{Q}{d_Q} u+d_{Q'} v$ and $b=\frac{N}{d_Q}+d_{Q'} Q$. We have $(b,Q)=\frac{Q}{d_Q}$ and $(b,Q')=d_{Q'}$ so that $(b,N)=\frac{Q}{d_Q} \cdot d_{Q'}$. Since $(a,\frac{Q}{d_Q})=(a,d_{Q'})=1$, it follows that $(a,b,N)=1$, whence the result.
\end{proof}

Let $d$ be a divisor of $N$. By Lemma \ref{lem WQ cusps}, there exists $Q$ such that $W_Q$ maps cusps of level $d$ to cusps of level $\delta=(d,\frac{N}{d})$. Note that $\delta^2 |N$. In view of the previous lemmas, Theorem \ref{main result} is reduced to showing that if $f$ is minimal by twist, then $e_f(\frac{1}{d})=1$ for every $d$ such that $d^2 | N$.

We now make use of the following idea : studying the behaviour of $f$ at $\frac{1}{d}$ amounts to studying the behaviour of $f(z+\frac{1}{d}) | W_N$ at infinity. More precisely, define $f_d(z) = f(z+\frac{1}{d})$. A direct computation shows that if $d^2 | N$ then $f_d \in S_2(\Gamma_1(N))$.

From now on, we fix an integer $d \geq 1$ such that $d^2 | N$ and we define $g_d = W_N(f_d) = \sum_{n \geq 1} b_{d,n} q^n \in S_2(\Gamma_1(N))$.

\begin{pro}\label{ord gd}
We have $e_f(\frac{1}{d})=\min \{n \geq 1 : b_{d,n} \neq 0\}$.
\end{pro}

\begin{proof}
Let $M=\frac{1}{\sqrt{N}} \begin{pmatrix} 1 & \frac{1}{d}\\ 0 & 1 \end{pmatrix} \begin{pmatrix} 0 & -1 \\ N & 0 \end{pmatrix} \in \SL_2(\R)$. We have $M(\infty)=\frac{1}{d}$ and $f|M=g_d$. Since $M^{-1} \Gamma_0(N) M \cap \begin{pmatrix} 1 & \R \\ 0 & 1 \end{pmatrix} = \begin{pmatrix} 1 & \Z \\ 0 & 1 \end{pmatrix}$, a uniformizing parameter at $[\frac{1}{d}] \in X_0(N)(\C)$ is given by $z \mapsto \exp(2\pi i M^{-1} z)$. It follows that $\ord_{\frac{1}{d}} \omega_f = \ord_{\infty} \omega_{g_d}$.
\end{proof}

Note that $e_f(1)=e_f(\infty)=1$. The case $d=2$ is also easily treated.

\begin{pro}
If $4|N$ then $e_f(\frac12)=1$.
\end{pro}

\begin{proof}
Since the Fourier expansion of $f$ involves only odd powers of $q$, we have $f(z+\frac12)=-f(z)$, so that $e_f(\frac12)=e_f(0)=1$.
\end{proof}

In Sections \ref{merel formula} and \ref{reduction local}, we reduce Theorem \ref{main result} to a purely local question on irreducible cuspidal representations of $\GL_2(\Q_p)$. We use a formula of Merel to express the $n$-th Fourier coefficient of $f$ at a given cusp of $X_0(N)$ as a certain sum of pseudo-eigenvalues of Atkin-Lehner involutions associated to twists of $f$. We then express these pseudo-eigenvalues as products of local epsilon factors. Another way of reducing Theorem \ref{main result} to a local statement would have been to express the Fourier coefficients of $f$ at a given cusp in terms of the Whittaker newform associated to $f$ as in \cite[Section 3.4.2]{nelson-pitale-saha}, and to use the formula for the Whittaker newform in \cite[Proposition 2.30]{saha:GL2}.

\section{Merel's formula}\label{merel formula}

In this section, we apply a formula of Merel \cite{merel:symbmanin} expressing the additive translate of a newform as a linear combination of certain twists of this newform. The related problem of computing the Fourier expansion of a newform at an arbitrary cusp has also been studied by Delaunay in his PhD thesis \cite[III.2]{delaunay:thesis}. Although Delaunay's results apply in the particular case considered here, we prefer to use Merel's formula since it does not assume that the newform is minimal by twist.

Let us first recall the notations of \cite{merel:symbmanin}. Let $\phi$ denote Euler's function. For any integer $m \geq 1$, let $\Sigma_m$ be the set of prime factors of $m$. For any Dirichlet character $\chi : (\Z/m\Z)^\times \to \C^\times$, the Gauss sum of $\chi$ is $\tau(\chi)=\sum_{a \in (\Z/m\Z)^\times} \chi(a) e^{2\pi i a/m}$, and the conductor of $\chi$ is denoted by $m_\chi$. For any newform $F$ of weight $k \geq 2$ on $\Gamma_1(M)$ and for any prime $p$, let $L_p(F,X)=1-a_p(F) X + a_{p,p}(F) p X^2 \in \C[X]$ be the inverse of the Euler factor of $F$ at $p$. If $T^+$ and $T^-$ are finite sets of prime numbers, we define
\begin{equation*}
F^{[T^+,T^-]} = F |_k \prod_{p \in T^+} L_p(F,p^{-k/2}\begin{pmatrix} p & 0 \\ 0 & 1 \end{pmatrix}) \prod_{p \in T^-} L_p(\bar F,p^{-k/2} \begin{pmatrix} 1 & 0 \\ 0 & p \end{pmatrix}).
\end{equation*}
There exists a unique newform $F \otimes \chi$ of weight $k$ and level dividing $\operatorname{lcm}(M,m^2)$ such that $a_p(F \otimes \chi)=a_p(F) \chi(p)$ for any prime $p \not\in \Sigma_{Mm}$.

Using \cite[(5)]{merel:symbmanin} with $\frac{n}{N}=\frac{1}{d}$, we get
\begin{equation}\label{fd}
f_d = \sum_{\chi}{\frac{\tau(\bar\chi)}{\phi(d)}}
({f\otimes{\chi}})^{[\Sigma_{d},\Sigma_{d}-\Sigma_{m_{\chi}}]}|\prod_{p \in \Sigma_{d/m_\chi}} P_{p}(\begin{pmatrix} p&0 \\ 0&1\end{pmatrix})
\end{equation}
where $\chi$ runs through the primitive Dirichlet characters of conductor $m_\chi$ dividing $d$, and the polynomial $P_p(X) \in \C[X]$ is given by
\begin{equation*}
P_p(X) = \begin{cases} -\bar\chi(p) & \textrm{if } a_p(f)=0, v_p(d)=1, v_p(m_\chi)=0,\\
(a_p(f)X)^{v_p(d/m_\chi)} & \textrm{otherwise}.
 \end{cases}
\end{equation*}

Since $a_p(f)=0$ for $p \in \Sigma_d$, the product over $p$ in (\ref{fd}) vanishes unless $(m_\chi,\frac{d}{m_\chi})=1$ and $\frac{d}{m_\chi}$ is squarefree. Let $S'(d)$ be the set of primitive Dirichlet characters $\chi$ such that $m_\chi | d$, $(m_\chi,\frac{d}{m_\chi})=1$ and $\frac{d}{m_\chi}$ is squarefree. Taking into account $L_p(f \otimes \chi,X)=1$ for $p \in \Sigma_{d/m_\chi}$, we get
\begin{equation}\label{fd2}
f_d = \sum_{\chi \in S'(d)} \frac{\tau(\bar\chi)}{\phi(d)} \Bigl(\prod_{p \in \Sigma_{d/m_\chi}} -\bar\chi(p)\Bigr)
 \cdot ({f\otimes{\chi}})^{[\Sigma_{m_{\chi}},\emptyset]}.
\end{equation}

\emph{From now on, we assume that $f$ is minimal by twist.} Then $f \otimes \chi$ has level exactly $N$ for every character $\chi$ of conductor dividing $d$, so that $({f\otimes{\chi}})^{[\Sigma_{m_{\chi}},\emptyset]} = f \otimes \chi$ for every $\chi \in S'(d)$.

Let $S(d)$ be the set of Dirichlet characters modulo $d$ induced by the elements of $S'(d)$. If $\chi' \in S'(d)$ induces $\chi \in S(d)$, then
\begin{equation*}
\tau(\chi)=\tau(\chi') \cdot \prod_{p \in \Sigma_{d/m_\chi}} -\chi'(p).
\end{equation*}
Thus $f_d$ can finally be rewritten
\begin{equation}\label{fd3}
f_d = \sum_{\chi \in S(d)} \frac{\tau(\bar\chi)}{\phi(d)} \cdot {f\otimes{\chi}}.
\end{equation}

We now apply $W_N$. We have $W_N(f \otimes \chi) = w(f \otimes \chi) \cdot f \otimes \bar\chi$, where $w(f \otimes \chi)$ is the pseudo-eigenvalue of $W_N$ at $f \otimes \chi$. It follows that
\begin{equation}\label{gd}
g_d = \sum_{\chi \in S(d)} \frac{\tau(\bar\chi)}{\phi(d)} \cdot w(f \otimes \chi) \cdot f \otimes \bar\chi.
\end{equation}
In particular, we get
\begin{equation}\label{bdn}
b_{d,n} =  \frac{a_n(f)}{\phi(d)} \sum_{\chi \in S(d)}\tau(\bar\chi) \cdot \bar\chi(n) \cdot w(f \otimes \chi) \qquad (n \geq 1).
\end{equation}
Note that $b_{d,n}=0$ whenever $(n,d)>1$, and that the inner sum in (\ref{bdn}) depends only on $n \mod{d}$. If $n=1$, then (\ref{bdn}) simplifies to
\begin{equation}\label{bd1}
b_{d,1} =  \frac{1}{\phi(d)} \sum_{\chi \in S(d)}\tau(\bar\chi) \cdot w(f \otimes \chi).
\end{equation}

\section{Reduction to a local computation}\label{reduction local}

In this section, we show that $b_{d,n}$ is a product of local terms depending only on the local automorphic representations associated to $f$, thereby reducing the non-vanishing of $b_{d,n}$ to a local question.

The basic observation is that if $d=p_1^{m_1} \ldots p_k^{m_k}$ is the prime factorization of $d$, then we have a natural bijection $S(d) \cong S(p_1^{m_1}) \times \cdots \times S(p_k^{m_k})$. Moreover $S(p)$ (resp. $S(p^m)$ with $m \geq 2$) is the set of Dirichlet characters modulo $p$ (resp. of conductor $p^m$). We will show that the summand in (\ref{bdn}) decomposes accordingly as a product of local terms. We shift to the adelic language, which is more convenient for our purposes.

Let $\A_\Q$ be the ring of ad\`eles of $\Q$. We view Dirichlet characters as characters of $\A_\Q^\times /\Q^\times$ as follows. We attach to $\chi \in S(d)$ the unique (continuous) character $\chi_\A : \A_\Q^\times /\Q^\times \to \C^\times$ such that for any $p \not\in \Sigma_d$, we have $\chi_\A(\varpi_p)=\chi(p)$, where $\varpi_p$ denotes a uniformizer of $\Q_p^\times \subset \A_\Q^\times$. For any $p \in \Sigma_d$, we denote by $\chi_p : \Q_p^\times \to \C^\times$ the $p$-component of $\chi_\A$. Letting $m_p=v_p(d)$, we have $\chi_p(1+p^{m_p}\Z_p)=1$. A word of caution is in order here: with the above convention, the map $\Z_p^\times/(1+p^{m_p}\Z_p) \to \C^\times$ induced by $\chi_p$ is the inverse of the $p$-component of $\chi$.

The level of a non-trivial additive character $\psi : \Q_p \to \C^\times$ is the unique integer $\ell \in \Z$ such that $\ker(\psi)=p^\ell \Z_p$. For any character $\psi : \Q_p \to \C^\times$ of level $m_p=v_p(d)$, we define the local Gauss sum of $\chi \in S(d)$ at $p$ by
\begin{equation}
\tau(\chi_p,\psi) = \sum_{x \in \Z_p^\times/(1+p^{m_p}\Z_p)} \chi_p(x) \psi(x).
\end{equation}

\begin{lem}\label{lem tauchi}
For any $n \in (\Z/d\Z)^\times$, there exist characters $\psi'_p : \Q_p \to \C^\times$ of respective levels $m_p=v_p(d)$ such that
\begin{equation}\label{tauchi}
\tau(\bar\chi) \cdot \bar\chi(n) = \prod_{p \in \Sigma_d} \tau(\chi_p,\psi'_p) \qquad (\chi \in S(d)).
\end{equation}
\end{lem}

\begin{proof}
Multiplying $\tau(\bar\chi)$ by $\bar\chi(n)$ only amounts to change the additive character in the definition of the Gauss sum of $\bar\chi$. The lemma now follows from the Chinese remainder theorem.
\end{proof}

Let $\pi_f$ be the automorphic representation of $\GL_2(\A_\Q)$ associated to $f$ \cite[\S 2.1]{loeffler-weinstein}. For any $\chi \in S(d)$, we have a canonical isomorphism $\pi_{f \otimes \chi} \cong \chi \pi_f$, where the latter representation is $g \mapsto \chi_\A(\det g) \pi_f(g)$. The $L$-function of $\pi_{f \otimes \chi}$ satisfies a functional equation \cite[Thm 11.1]{jacquet-langlands}
\begin{equation}\label{eq func}
L(\pi_{f \otimes \chi},s) = \epsilon(\pi_{f \otimes \chi},s) L(\pi_{f \otimes \bar\chi},1-s),
\end{equation}

Fix an additive character $\psi = \prod_v \psi_v : \A_\Q/\Q \to \C^\times$ such that $\psi_p$ has level one for every $p \in \Sigma_d$. By \cite[\S 11]{jacquet-langlands}, we have
\begin{equation}
\epsilon(\pi_{f \otimes \chi},s) = \prod_v \epsilon(\pi_{f \otimes \chi,v},s,\psi_v)
\end{equation}
where $v$ runs through the places of $\Q$, and $\pi_{f \otimes \chi,v}\cong \chi_v \pi_{f,v}$ denotes the local component of $f \otimes \chi$ at $v$. The quantity $\epsilon(\pi_{f \otimes \chi,v},s,\psi_v)$ is the Godement-Jacquet local constant of $\pi_{f \otimes \chi}$.

For any character $\chi$ of $\Q_p^\times$, we let $\widetilde\chi$ be the unique character of $\Q_p^\times$ such that $\widetilde{\chi}(p)=1$ and $\widetilde{\chi}|_{\Z_p^\times} = \chi |_{\Z_p^\times}$. The following proposition shows that $w(f \otimes \chi)$ can be written as a product of local constants.

\begin{pro}\label{pro wfchi}
There exist a constant $C \in \C^\times$ and an element $a \in (\Z/d\Z)^\times$, depending on $f$ and $\psi$ but not on $\chi$, such that
\begin{equation}\label{wfchi}
w(f \otimes \chi) = C \cdot \chi(a) \prod_{p \in \Sigma_d} \epsilon(\widetilde{\chi}_p \pi_{f,p},\frac12,\psi_p) \qquad (\chi \in S(d)).
\end{equation}
\end{pro}

\begin{proof}
Let $L(f \otimes \chi,s)$ be the usual $L$-function of $f \otimes \chi$. It relates to the automorphic $L$-function by $L(\pi_{f \otimes \chi},s-\frac12) = (2\pi)^{-s} \Gamma(s) L(f \otimes \chi,s)$. Comparing (\ref{eq func}) with the usual functional equation yields
\begin{equation}
w(f \otimes \chi) = -N^{s-\frac12} \epsilon(\pi_{f \otimes \chi},s).
\end{equation}
By \cite[Prop 5.21, Thm 6.16]{gelbart}, we have $\epsilon(\pi_{f \otimes \chi,\infty},s,\psi_\infty)=-1$, so we get
\begin{equation}
w(f \otimes \chi) =  \prod_{p \in \Sigma_N} \epsilon(\pi_{f \otimes \chi,p},\frac12,\psi_p) = \prod_{p \in \Sigma_N} \epsilon(\chi_p \pi_{f,p},\frac12,\psi_p).
\end{equation}
It follows from the definition of the epsilon factor \cite[\S 24.2]{bushnell-henniart} that there exists an integer $b_p \in \Z$ not depending on $\chi_p$ such that for every unramified character $\omega_p : \Q_p^\times \to \C^\times$, we have
\begin{equation}
\epsilon(\omega_p \chi_p \pi_{f,p},s,\psi_p) = \omega_p(p^{b_p}) \epsilon(\chi_p \pi_{f,p},s,\psi_p).
\end{equation}
Choosing $\omega_p$ such that $\omega_p \chi_p = \widetilde{\chi}_p$, and noting that
\begin{equation*}
\prod_{p \in \Sigma_N} \bar\omega_p(p^{b_p}) = \prod_{p \in \Sigma_N} \chi_p(p^{b_p})
\end{equation*}
may be written $\chi(a)$ with $a \in (\Z/d\Z)^\times$ not depending on $\chi$, we get the result by taking $C=\prod_{p \in \Sigma_N-\Sigma_d} \epsilon(\pi_{f,p},\frac12,\psi_p)$.
\end{proof}

The map $\chi \mapsto (\widetilde{\chi}_p)_{p \in \Sigma_d}$ provides a bijection $S(d) \cong \prod_{p \in \Sigma_d} \widetilde{S}(p^{m_p})$, where $\widetilde{S}(p^m)$ is the set of characters $\omega : \Q_p^\times \to \C^\times$ such that $\omega(p)=1$, $\omega(1+p^m\Z_p)=1$, and $\omega(1+p^{m-1}\Z_p) \neq 1$ if $m \geq 2$. Putting together the formulas (\ref{bdn}), (\ref{tauchi}) and (\ref{wfchi}), we get
\begin{equation}\label{bdn product}
b_{d,n} = \frac{C}{\phi(d)} \cdot a_n(f) \prod_{p \in \Sigma_d} \sum_{\chi_p \in \widetilde{S}(p^{m_p})} \tau(\chi_p,\psi'_p) \cdot \epsilon(\chi_p \pi_{f,p},\frac12,\psi_p)
\end{equation}
for some characters $\psi'_p : \Q_p \to \C^\times$ of respective levels $m_p=v_p(d)$.

Let $E$ be an elliptic curve over $\Q$ of conductor $N$. Assume that the newform $f_E$ associated to $E$ is minimal by twist. For any prime $p$ such that $p^2 | N$, the local component $\pi_{f_E,p}$ is an irreducible cuspidal representation of $\GL_2(\Q_p)$ by \cite[Prop. 2.8]{loeffler-weinstein}. Theorem \ref{main result} is thus reduced to the following purely local statement.

\begin{thm}\label{main result local}
Let $\pi$ be an irreducible cuspidal representation of\linebreak $\GL_2(\Q_p)$ with trivial central character and of conductor $p^n$ with $n \geq 2$. Assume that $\pi$ is twist minimal, and that $\pi$ can be realized over $\Q$. Let $m$ be an integer such that $1 \leq m \leq n/2$. Then for any characters $\psi,\psi' : \Q_p \to \C^\times$ of respective levels $1$ and $m$, we have
\begin{equation}\label{sum pi pm}
\sum_{\chi \in \widetilde{S}(p^m)} \tau(\chi,\psi') \epsilon(\chi \pi,\frac12,\psi) \neq 0.
\end{equation}
\end{thm}

\begin{remark}
There are examples of newforms $f$ of weight $2$ on $\Gamma_0(N)$ such that the local component $\pi_{f,p}$ is cuspidal and twist minimal, but the sum (\ref{sum pi pm}) vanishes (see Remark \ref{rk 625}). Therefore the assumption that $\pi$ can be realized over $\Q$ is necessary.
\end{remark}

\section{Cuspidal inducing data}

In this section we recall how cuspidal representations of $\GL_2(\Q_p)$ can be described in terms of cuspidal inducing data, and we recast Theorem \ref{main result local} in terms of this data using a formula of Bushnell for the local constant.

Let $\pi$ be an irreducible cuspidal representation of $G=\GL_2(\Q_p)$ with trivial central character and of conductor $p^n$ with $n \geq 2$. By the classification theorem \cite[15.5,15.8]{bushnell-henniart}, the representation $\pi$ is induced by a cuspidal datum : there exist a maximal compact-mod-center subgroup $K$ of $G$, and an irreducible complex representation $\xi$ of $K$, such that $\pi \cong  \cInd_K^G \xi$, where $\cInd$ denotes compact induction.

Since $\pi$ has trivial central character, the restriction of $\xi$ to the center $Z=\Q_p^\times$ of $G$ is trivial, and since $K/Z$ is compact, $\xi$ is finite-dimensional.  The contragredient of $\xi$ is defined by $\check\xi(k)=\xi(k^{-1})^*$. Finally, note that $\chi \pi \cong \cInd_K^G (\chi \xi)$ for any character $\chi : \Q_p^\times \to \C^\times$.

There are two maximal compact-mod-center subgroups of $G$ up to conjugacy, namely $K'=p^\Z \cdot \GL_2(\Z_p)$ and $K''=\begin{pmatrix} 0 & 1 \\ p & 0 \end{pmatrix}^\Z \cdot \begin{pmatrix} \Z_p^\times & \Z_p \\ p\Z_p & \Z_p^\times \end{pmatrix}$. They are equipped with a canonical decreasing sequence of compact normal subgroups $(K_n)_{n \geq 0}$, which are defined as follows.

If $K=K'$ then $K_0=\GL_2(\Z_p)$ and $K_n=1+p^nM_2(\Z_p)$ for any $n \geq 1$. Note that $K_0/K_n \cong \GL_2(\Z/p^n\Z)$.

If $K=K''$ then $K_0=\begin{pmatrix} \Z_p^\times & \Z_p \\ p\Z_p & \Z_p^\times \end{pmatrix}$ and $K_n=1+\mathfrak{P}^n$ for any $n \geq 1$, where $\mathfrak{P} = \begin{pmatrix} p\Z_p & \Z_p \\ p\Z_p & p\Z_p \end{pmatrix}$.

The \emph{conductor} $r(\xi)$ of $\xi$ is the least integer $r \geq 1$ such that $\xi(K_r)=1$. The relation between the conductors of $\pi$ and $\xi$ is as follows \cite[A.3]{breuil-mezard}.

If $n=2m$ is even, we are in the \emph{unramified case}: we have $K=K'$ and $r(\xi)=m$, and we define $c=p^{1-m} \cdot I_2 \in K$.

If $n=2m+1$ is odd, we are in the \emph{ramified case}: we have $K=K''$ and $r(\xi)=2m$, and we define $c=\begin{pmatrix} 0 & -p^{-m} \\ p^{1-m} & 0 \end{pmatrix} \in K$.

The proof of Theorem \ref{main result local} relies on the following explicit formula, due to Bushnell, for the Godement-Jacquet local constant of $\pi$.

\begin{thm}\cite[25.2 Thm]{bushnell-henniart}\label{bushnell's formula}
Let $r$ be the conductor of $\xi$, and let $m=\lfloor \frac{n}{2} \rfloor$. If $\psi : \Q_p \to \C^\times$ is a character of level one, then
\begin{equation}
\sum_{x \in K_0/K_r} \psi(\tr(cx)) \check{\xi}(cx) = p^{2m} \epsilon(\pi,\frac12,\psi) \cdot \id.
\end{equation}
\end{thm}

We now express the sum of local constants appearing in Theorem \ref{main result local} in terms of $\xi$.

\begin{pro}\label{somme chi}
Let $r$ be the conductor of $\xi$, and let $m=\lfloor \frac{n}{2} \rfloor$. Let $k$ be an integer such that $1 \leq k \leq m$. For any characters $\psi,\psi' : \Q_p \to \C^\times$ of respective levels $1$ and $k$, the sum
\begin{equation}
\sum_{\chi \in \widetilde{S}(p^k)} \tau(\chi,\psi') \epsilon(\chi \pi,\frac12,\psi)
\end{equation}
is the unique eigenvalue of the scalar endomorphism
\begin{equation}\label{somme xi}
\frac{p-1}{p^{2m-k+1}} \sum_{x \in K_0/K_r} \psi(\tr(cx)) \psi'(\det x) \check{\xi}(cx).
\end{equation}
\end{pro}

\begin{proof}
Let $\chi \in \widetilde{S}(p^k)$. Since $\pi$ is minimal by twist, we have $r(\chi \xi)=r$ and Theorem \ref{bushnell's formula} gives
\begin{equation}\label{somme chi 1}
\sum_{x \in K_0/K_r} \psi(\tr(cx)) \bar\chi(\det(cx)) \check{\xi}(cx) = p^{2m} \epsilon(\chi\pi,\frac12,\psi) \cdot \id.
\end{equation}
Because $\det(c)$ is a power of $p$, we have $\bar\chi(\det(c))=1$. Multiplying the left hand side of (\ref{somme chi 1}) by $\tau(\chi,\psi')$ and summing over $\chi$, we get
\begin{align}
\nonumber & \sum_{\chi \in \widetilde{S}(p^k)} \sum_{y \in (\Z/p^k\Z)^\times} \chi(y) \psi'(y) \sum_{x \in K_0/K_r} \psi(\tr(cx)) \bar\chi(\det x) \check{\xi}(cx)\\
\label{somme chi 3} = & \sum_{x \in K_0/K_r} \psi(\tr(cx)) \check{\xi}(cx) \sum_{y \in (\Z/p^k\Z)^\times} \psi'(y) \sum_{\chi \in S(p^k)} \chi(y) \bar\chi(\det x).
\end{align}
Let $C(p^k)$ be the set of all Dirichlet characters modulo $p^k$. For $a \in (\Z/p^k\Z)^\times$, the sum $\sum_{\chi \in C(p^k)} \chi(a)$ equals $p^{k-1}(p-1)$ if $a=1$, and $0$ otherwise. So for $k=1$, (\ref{somme chi 3}) simplifies to
\begin{equation}\label{somme chi 4}
(p-1)\sum_{x \in K_0/K_r} \psi(\tr(cx)) \psi'(\det x) \check{\xi}(cx).
\end{equation}
If $k \geq 2$ then $\widetilde{S}(p^k)=C(p^k)-C(p^{k-1})$ so that (\ref{somme chi 3}) can be written
\begin{align}
\nonumber & p^{k-1}(p-1) \sum_{x \in K_0/K_r} \psi(\tr(cx)) \psi'(\det x) \check{\xi}(cx)\\
\label{somme chi 5}  & - p^{k-2}(p-1)\sum_{x \in K_0/K_r} \psi(\tr(cx)) \Biggl(\sum_{\substack{y \in (\Z/p^k\Z)^\times\\ y \equiv \det x \pod{p^{k-1}}}} \psi'(y)\Biggr) \check{\xi}(cx).
\end{align}
Since $\psi'$ has level $k$, the inner sum over $y$ vanishes. In all cases, this gives the proposition as stated.
\end{proof}

\begin{remark}\label{remark bdn}
The formula (\ref{bdn product}), together with Proposition \ref{somme chi}, leads to an explicit formula for the Fourier expansion of $f$ at an arbitrary cusp of $X_0(N)$ purely in terms of the local components of $f$, and may be of independent interest.
\end{remark}

\begin{definition}
For any characters $\psi,\psi' : \Q_p \to \C^\times$ of respective levels $1$ and $m$, we define $T(\xi,\psi,\psi')$ to be the endomorphism
\begin{equation}\label{Txi}
T(\xi,\psi,\psi') = \sum_{x \in K_0/K_{r(\xi)}} \psi(\tr(cx)) \psi'(\det x) \check{\xi}(x).
\end{equation}
\end{definition}
In order to establish Theorem \ref{main result local}, it suffices, thanks to Proposition \ref{somme chi}, to show that $T(\xi,\psi,\psi') \neq 0$. We prove this in the following sections, distinguishing the unramified and ramified cases.

\section{The unramified case}\label{sec unramified}

In this section we assume $n=2m$ with $m \geq 1$, so that $c=p^{1-m} \cdot I_2$. Note that $\psi(\tr(cx)) = \psi(p^{1-m} \tr x)$ and $a \mapsto \psi(p^{1-m}a)$ is a character of level $m$. So we fix characters $\psi,\psi' : \Q_p \to \C^\times$ of respective levels $m,m'$ with $1 \leq m' \leq m$, and we wish to prove that
\begin{equation}\label{Txi unram}
T(\xi,\psi,\psi') := \sum_{x \in \GL_2(\Z/p^m\Z)}  \psi(\tr x) \psi'(\det x)\check{\xi}(x)
\end{equation}
is non-zero. Assuming the contrary, for every $y \in \GL_2(\Z/p^m\Z)$ we have
\begin{align*}
0 = \check{\xi}(y^{-1}) T(\xi,\psi,\psi') & = \sum_{x \in \GL_2(\Z/p^m \Z)}  \psi(\tr x) \psi'(\det x) \check{\xi}(y^{-1} x) \\
& = \sum_{x \in \GL_2(\Z/p^m \Z)}  \psi(\tr(yx)) \psi'(\det (yx)) \check{\xi}(x).
\end{align*}
Taking $y=\begin{pmatrix} 1 & t \\ 0 & 1 \end{pmatrix}$ with $t \in \Z/p^m\Z$ and writing $x=\begin{pmatrix}\alpha & \beta \\ \gamma & \delta\end{pmatrix}$, we have $\tr(yx)=\tr(x)+\gamma t$. We get
\begin{equation}\label{eq 1}
\sum_{x \in \GL_2(\Z/p^m\Z)} \psi(\gamma t) \psi(\tr x) \psi'(\det x)\check{\xi}(x) = 0.
\end{equation}
For any $c_0 \in \Z/p^m\Z$, let $B(c_0) \subset \GL_2(\Z/p^m\Z)$ be the set of matrices of the form $x = \begin{pmatrix} * & * \\ c_0 & * \end{pmatrix}$. Since (\ref{eq 1}) is true for every $t \in \Z/p^m\Z$, we get
\begin{equation}
\sum_{x \in B(c_0)} \psi(\tr x) \psi'(\det x)\check{\xi}(x) = 0 \qquad (c_0 \in \Z/p^m\Z).
\end{equation}
Fix $c_0 \in (\Z/p^m\Z)^\times$. Then every matrix $x \in B(c_0)$ may be written uniquely in the form $x=\begin{pmatrix} 0 & 1 \\ c_0 & 0 \end{pmatrix} \begin{pmatrix} 1 & 0 \\ a & 1 \end{pmatrix} \begin{pmatrix} 1 & b \\ 0 & d \end{pmatrix}$ with $a,b \in \Z/p^m\Z$ and $d \in (\Z/p^m\Z)^\times$. We have $\tr x = a+bc_0$ and $\det x = -dc_0$, so that
\begin{equation}\label{eq 2}
\sum_{\substack{a,b \in \Z/p^m\Z \\ d \in (\Z/p^m\Z)^\times}} \psi(a+bc_0) \psi'(-dc_0) \check{\xi} \begin{pmatrix} 1 & 0 \\ a & 1 \end{pmatrix} \check{\xi} \begin{pmatrix} 1 & b \\ 0 & d \end{pmatrix} = 0 \qquad (c_0 \in (\Z/p^m\Z)^\times).
\end{equation}

We now make use of further properties of the representation $\xi$, for which we refer to \cite{loeffler-weinstein}. Let $V$ be the space of $\check{\xi}$. By \cite[Thm 3.5]{loeffler-weinstein}, the restriction of $\check{\xi}$ to $N=\begin{pmatrix} 1 & \Z_p \\ 0 & 1\end{pmatrix}$ is isomorphic to the direct sum of the additive characters of $\Z_p$ of level $m$, each character appearing with multiplicity $1$. We denote by $V=\bigoplus_\chi V(\chi)$ this direct sum decomposition. Moreover, by the proof of \cite[Thm 3.6]{loeffler-weinstein}, the representation $\pi$ admits a new vector: there exists $v \in V-\{0\}$ which is fixed by all diagonal matrices of $K$. Since $Nv$ spans $V$, the components $v_\chi$ of $v$ with respect to the above decomposition are nonzero. Note that the diagonal matrix $\delta = \begin{pmatrix} 1 & 0 \\ 0 & d \end{pmatrix}$ sends $V(\chi)$ to $V(\chi_d)$, where $\chi_d$ denotes the character $t \mapsto \chi(dt)$. Since $\delta$ fixes $v$, we get $\delta(v_\chi)=v_{\chi_d}$. It follows that
\begin{equation}
\check{\xi}\begin{pmatrix} 1 & b \\ 0 & d \end{pmatrix} v_\chi = \check{\xi}\begin{pmatrix} 1 & 0 \\ 0 & d \end{pmatrix}\check{\xi}\begin{pmatrix} 1 & b \\ 0 & 1 \end{pmatrix} v_\chi = \chi(b) v_{\chi_d}.
\end{equation}
Let $p'_\psi = \sum_{a \in \Z/p^m\Z} \psi(a) \check{\xi}\begin{pmatrix} 1 & 0 \\ a & 1 \end{pmatrix}$. Since the matrix $\begin{pmatrix} 1 & 0 \\ a & 1 \end{pmatrix}$ is conjugate to $\begin{pmatrix} 1 & a \\ 0 & 1 \end{pmatrix}$ and $\psi$ has level $m$, we see that $p'_\psi$ is a projector of rank $1$. Evaluating (\ref{eq 2}) at $v_\chi$, we get
\begin{equation}\label{eq 3}
p'_\psi \sum_{\substack{b \in \Z/p^m\Z \\ d \in (\Z/p^m\Z)^\times}} \psi(bc_0) \psi'(-dc_0) \chi(b) v_{\chi_d} = 0 \qquad (c_0 \in (\Z/p^n\Z)^\times).
\end{equation}
The sum over $b$ is zero unless $\chi(t)=\psi(-c_0 t)$, in which case (\ref{eq 3}) simplifies to
\begin{equation}\label{eq 4}
p'_\psi \sum_{d \in (\Z/p^m\Z)^\times} \psi'(-dc_0) v_{\psi_{-c_0 d}} = 0 \qquad (c_0 \in (\Z/p^m\Z)^\times).
\end{equation}
Let $\lambda \in \Z/p^m\Z$ be such that $\psi'=\psi_\lambda$ (we have $\lambda \neq 0$ since $m' \geq 1$). Then (\ref{eq 4}) implies
\begin{equation}\label{eq 5}
p'_\psi \sum_\chi \chi(\lambda) v_\chi = 0
\end{equation}
where the sum runs over all primitive characters of $\Z/p^m\Z$. This can be rewritten as
\begin{equation}\label{eq 6}
p'_\psi \check{\xi} \begin{pmatrix} 1 & \lambda \\ 0 & 1 \end{pmatrix} v = 0.
\end{equation}
We now make use of the assumption that the representation $\xi$ can be realized over $\Q$. The subspace of $V$ fixed by all diagonal matrices is one-dimensional and is clearly rational. The new vector $v$ can thus be chosen to be rational. By taking Galois conjugates, if (\ref{eq 6}) holds for one particular value of $\psi$, then it holds for all $\psi$ of level $m$. But the projectors $p'_\psi$ add up to the identity of $V$, so that we get $v=0$, a contradiction.

\begin{remark}\label{rk 625}
If we don't assume that $\xi$ can be realized over $\Q$, then it may happen that $T(\xi,\psi,\psi')=0$ even if $\xi$ is twist minimal. We found examples of this phenomenon already in the case $p=5$ and $m=2$. There is a cuspidal representation $\xi$ of $\GL_2(\Z/25\Z)$ of dimension 20 with coefficients in $\Q(\sqrt{5})$ such that $T(\xi,\psi,\psi_\lambda)=0$ if $\lambda \equiv \pm 1 \pmod{5}$. There is a newform $f$ of weight $2$ on $\Gamma_0(625)$ with coefficients in $\Q(\sqrt{5})$ whose Fourier expansion begins with
\begin{equation*}
f = q + \left(\frac{-1-\sqrt{5}}{2}\right) q^2 + \left(\frac{-3-\sqrt{5}}{2}\right) q^3 + \left(\frac{-1+\sqrt{5}}{2}\right) q^4 \ldots
\end{equation*}
such that $\pi_{f,5}$ is induced from $\xi$. The newform $f$ is twist minimal, and the formula (\ref{bdn product}) together with Proposition \ref{somme chi} implies that $\omega_f$ vanishes at the cusps $\lambda/25$ with $\lambda \equiv \pm 1 \pmod{5}$. Hao Chen has checked numerically that it is indeed the case.
\end{remark}

\section{The ramified case}\label{sec ramified}

In this section we assume $n=2m+1$ with $m \geq 1$, so that $c=\begin{pmatrix} 0 & -p^{-m}\\ p^{1-m} & 0 \end{pmatrix}$. Note that $\psi(\tr(cx))= \psi(p^{1-m} \tr' x)$ where the function $\tr' : K_0 \to \Z_p$ is defined by $\tr' \begin{pmatrix} \alpha & \beta \\ p\gamma & \delta \end{pmatrix} = \beta-\gamma$. So we fix characters $\psi,\psi' : \Q_p \to \C^\times$ of respective levels $m,m'$ with $1 \leq m' \leq m$, and we wish to prove that
\begin{equation}\label{Txi ram}
T(\xi,\psi,\psi') := \sum_{x \in K_0/K_{2m}}  \psi(\tr' x) \psi'(\det x)\check{\xi}(x)
\end{equation}
is non-zero. Assume the contrary.

We have explicitly
\begin{equation*}
K_\ell = \begin{pmatrix} 1+p^{\lceil \frac{\ell}{2} \rceil}\Z_p & p^{\lfloor \frac{\ell}{2} \rfloor}\Z_p \\
p^{\lfloor \frac{\ell}{2}\rfloor +1}\Z_p & 1+p^{\lceil \frac{\ell}{2} \rceil}\Z_p \end{pmatrix} \qquad (\ell \geq 1).
\end{equation*}
Moreover, we have an isomorphism of groups
\begin{align*}
K_m/K_{2m} & \xrightarrow{\cong} (\Z/p^{\lfloor \frac{m}{2}\rfloor}\Z)^2 \oplus (\Z/p^{\lceil \frac{m}{2}\rceil}\Z)^2\\
\begin{pmatrix} 1+p^{\lceil \frac{m}{2}\rceil} \alpha & p^{\lfloor \frac{m}{2}\rfloor} \beta\\
p^{\lfloor \frac{m}{2}\rfloor+1} \gamma & 1+p^{\lceil \frac{m}{2}\rceil} \delta
\end{pmatrix} & \mapsto (\alpha,\delta,\beta,\gamma).
\end{align*}

Let $y \in K_m$. Multiplying $T(\xi,\psi,\psi')$ on the right by $\check\xi(y^{-1})$, we get
\begin{equation}\label{ram somme2}
\sum_{x \in K_0/K_{2m}} \psi(\tr'(xy)) \psi'(\det(xy)) \check\xi(x)=0.
\end{equation}
If we fix $x \in K_0$, then the map $\Phi_x : K_m/K_{2m} \to \C^\times$ defined by
\begin{equation}
\psi(\tr'(xy)) \psi'(\det(xy)) = \psi(\tr' x) \psi'(\det x) \Phi_x(y) \qquad (y \in K_m)
\end{equation}
is a character which depends only on the coset $xK_m$.

\begin{lem}\label{lem Phix}
The characters $(\Phi_x)_{x \in K_0/K_m}$ are pairwise distinct.
\end{lem}

\begin{proof}
If $x=\begin{pmatrix} a & b \\ pc & d \end{pmatrix} \in K_0$ and $y=\begin{pmatrix} 1+s & t \\ pu & 1+v \end{pmatrix} \in K_m$, an explicit computation gives
\begin{equation}\label{Phix explicite}
\Phi_x(y) = \psi(at+bv-cs-du) \psi'((ad-pbc)(s+v)).
\end{equation}
Let $x'=\begin{pmatrix} a' & b' \\ pc' & d' \end{pmatrix} \in K_0$ such that $\Phi_x = \Phi_{x'}$. By (\ref{Phix explicite}), we already get $a,d \equiv a',d'  \pod{p^{\lceil \frac{m}{2}\rceil}}$. Let $\lambda \in (\Z/p^{m'}\Z)^\times$ be the unique element such that $\psi'(1)=\psi(p^{m-m'} \lambda)$. It remains to prove that the map
\begin{align}\label{map had}
h_{a,d} : (\Z/p^{\lfloor \frac{m}{2}\rfloor} \Z)^2 & \to (\Z/p^{\lfloor \frac{m}{2}\rfloor} \Z)^2\\
\nonumber (b,c) & \mapsto (b+p^{m-m'}\lambda (ad-pbc),-c+p^{m-m'}\lambda (ad-pbc))
\end{align}
is injective. Assume $h_{a,d}(b,c)=h_{a,d}(b',c')$. Then $b-p^{m-m'+1}\lambda bc \equiv b'-p^{m-m'+1}\lambda b'c' \pod{p^{\lfloor \frac{m}{2}\rfloor}}$ and $c+p^{m-m'+1}\lambda bc \equiv c'+p^{m-m'+1}\lambda b'c' \pod{p^{\lfloor \frac{m}{2}\rfloor}}$. In particular $b,c \equiv b',c' \pod{p}$ and an easy induction gives $b,c \equiv b',c' \pod{p^{\lfloor \frac{m}{2}\rfloor}}$.
\end{proof}

Fix $x_0 \in K_0$. If we multiply (\ref{ram somme2}) by $\bar\Phi_{x_0}(y)$ and sum over $y \in K_m/K_{2m}$, we get
\begin{equation}\label{ram somme3}
\sum_{y \in K_m/K_{2m}} \bar\Phi_{x_0}(y) \sum_{x \in K_0/K_{2m}} \psi(\tr' x) \psi'(\det x) \Phi_x(y) \check\xi(x)=0.
\end{equation}
According to Lemma \ref{lem Phix}, this simplifies to
\begin{equation}\label{ram somme4}
\sum_{x \in x_0 K_m/K_{2m}} \psi(\tr' x) \psi'(\det x) \check\xi(x)=0.
\end{equation}
In other words, for every $x_0 \in K_0$ we have
\begin{equation}\label{ram somme5}
\sum_{y \in K_m/K_{2m}} \Phi_{x_0}(y) \check\xi(y)=0.
\end{equation}

Fix $a_0,d_0 \in (\Z/p^{\lceil \frac{m}{2}\rceil}\Z)^\times$. We sum (\ref{ram somme5}) over all matrices $x_0 \in K_0/K_m$ of the form $x_0=\begin{pmatrix} a_0 & * \\ * & d_0 \end{pmatrix}$. Letting $y=\begin{pmatrix} 1+s & t \\ pu & 1+v \end{pmatrix}$, we compute
\begin{equation}
\sum_{x_0} \Phi_{x_0}(y) = \sum_{b_0,c_0 \in \Z/p^{\lfloor \frac{m}{2}\rfloor}\Z} \psi(a_0t+d_0u) \psi(h_{a_0,d_0}(b_0,c_0) \cdot (v,s))
\end{equation}
where $h_{a_0,d_0}$ is the map of (\ref{map had}). Since $h_{a_0,d_0}$ is bijective, we get
\begin{align*}
\sum_{x_0} \Phi_{x_0}(y) & = \psi(a_0t+d_0u) \sum_{b_0,c_0 \in \Z/p^{\lfloor \frac{m}{2}\rfloor}\Z} \psi(b_0v+c_0s)\\
& = \begin{cases} p^{2\lfloor \frac{m}{2}\rfloor} \psi(a_0t+d_0u) & \textrm{if } s\equiv v \equiv 0 \pod{p^m}\\
0 & \textrm{otherwise.}
\end{cases}
\end{align*}
So for any $a_0,d_0 \in (\Z/p^{\lceil \frac{m}{2}\rceil}\Z)^\times$, we get
\begin{equation}\label{ram somme6}
\sum_{t,u \in p^{\lfloor \frac{m}{2}\rfloor}\Z/p^m\Z}  \psi(a_0t+d_0u) \check\xi \begin{pmatrix} 1 & t \\ pu & 1 \end{pmatrix}=0.
\end{equation}
As in section \ref{sec unramified}, the restriction of $\xi$ to $N=\begin{pmatrix} 1 & \Z_p \\ 0 & 1 \end{pmatrix}$ is isomorphic to the direct sum of the characters of $\Z_p$ of level $m$. Conjugating by the matrix $\begin{pmatrix} 0 & 1 \\ p & 0 \end{pmatrix}$, the same is true for the restriction of $\xi$ to $N' = \begin{pmatrix} 1 & 0 \\ p\Z_p & 1 \end{pmatrix}$. For any $t,u \in p^{\lfloor \frac{m}{2}\rfloor}\Z/p^m\Z$, the matrices $\begin{pmatrix} 1 & t \\ 0 & 1 \end{pmatrix}$ and $\begin{pmatrix} 1 & 0 \\ pu & 1 \end{pmatrix}$ commute. By simultaneous diagonalization, there exists a nonzero vector $v$ in the space of $\check\xi$ and primitive characters $\omega,\omega' : \Z/p^{\lceil \frac{m}{2}\rceil}\Z \to \C^\times$ such that
\begin{equation}
\check\xi \begin{pmatrix} 1 & p^{\lfloor \frac{m}{2}\rfloor} t \\ p^{\lfloor \frac{m}{2}\rfloor+1} u & 1 \end{pmatrix} v = \omega(t) \omega'(u) v \qquad (t,u \in \Z/p^{\lceil \frac{m}{2}\rceil}\Z).
\end{equation}
We may write the characters $\omega, \omega'$ as $\omega(t) = \bar\psi(p^{\lfloor \frac{m}{2}\rfloor}a_0t)$ and $\omega'(u) = \bar\psi(p^{\lfloor \frac{m}{2}\rfloor}d_0u)$ for some $a_0,d_0 \in (\Z/p^{\lceil \frac{m}{2}\rceil}\Z)^\times$. For this choice of $a_0,d_0$, the identity (\ref{ram somme6}) evaluated at $v$ gives a contradiction. This finishes the proof of Theorem \ref{main result}.

\section{Numerical investigations}\label{numerical}

We now report on the computations which led to Theorem \ref{main result}. For all elliptic curves of conductor $\leq 2000$, we computed all the ramification indices at the cusps of the modular parametrizations using \textsc{Pari/GP} \cite{pari250}. Since we have no theoretical formula for the ramification index in general, we just compared numerically $\log |f_E|$ and $\log |q|$ in the neighborhood of the given cusp. This method is not rigorous, but it gives good results in practice. We ended up with a list of 745 isogeny classes of elliptic curves for which the modular parametrization seemed to ramify at some cusp. We then observed and checked, with the help of \textsc{Magma} \cite{magma}, that for each curve in this list, the associated newform was not minimal by twist.

In Table 1 below, we give all instances of ramified cusps for elliptic curves of conductor $\leq 200$ (we restrict to the cusps $[\frac1d]$ with $d^2|N$). In the last column, we indicate a minimal twist for the newform (it need not be unique, for example $96a$ has quadratic twist $96b$). Furthermore, a minimal twist need not have trivial character. For example, the minimal twist of $162b$ and $162c$ is a newform of level 18 and non-trivial character, which we just denote by ``18''.

\begin{table}[h]
\begin{center}
\begin{tabular}{c|c|c|c}
\textrm{Isogeny class} & $d$ & $e_\varphi(\frac1d)$ & \textrm{Minimal twist} \\
\hline
$48a$  & 4 & 2 & $24a$\\
$64a$  & 8 & 2 & $32a$ \\
$80a$  & 4  & 2 & $40a$\\
$80b$  &  4  & 4 & $20a$\\
$112a$ &  4  & 2 & $56b$\\
$112b$ &  4  & 2 & $56a$\\
$112c$ &  4  & 4 & $14a$\\
$144a$ & $\Bigl\{\!\!\begin{array}{c} 4 \\ 12 \end{array}$  & $\begin{array}{c} 4 \\ 4 \end{array}$ & $36a$\\
$144b$ &  $\Bigl\{\!\!\begin{array}{c} 4 \\ 12 \end{array}$  & $\begin{array}{c} 2 \\ 2 \end{array}$ & $24a$\\
$162b$ &  9  & 3 & 18\\
$162c$  &  9  & 3 & 18\\
$176a$ &  4  & 2 & $88a$\\
$176b$ &  4  & 4 & $11a$\\
$176c$ &  4  & 4 & $44a$\\
$192a$ &  8  & 2 & $96a$\\
$192b$ &  8  & 2 & $96a$\\
$192c$ &  8  & 4 & $24a$\\
$192d$ &  8  & 4 & $24a$
\end{tabular}
\end{center}
\caption{Ramified cusps for conductors $\leq 200$}
\end{table}

Note also that being minimal by twist is far from being a necessary condition in order for the modular parametrization to be unramified at the cusps. For example, the isogeny class $45a$, which is a twist of $15a$, has a modular parametrization which is unramified at the cusps.

In all cases we computed, the following properties seem to hold :

\begin{enumerate}
\item If $e_\varphi(\frac{1}{d})$ is even then $v_2(d) \in \{2,3,4\}$ and $v_2(N) =2v_2(d)$;
\item If  $e_\varphi(\frac{1}{d})$ is divisible by $8$ then $v_2(d)=4$ and $v_2(N)=8$;
\item If  $e_\varphi(\frac{1}{d})$ is divisible by $3$ then $v_3(d)=2$ and $v_3(N)=4$.
\end{enumerate}

These observations are consistent with the following theorem of Atkin and Li \cite[Thm 4.4.i)]{atkin-li} : if $f \in S_2(\Gamma_0(N))$ is a newform and $v_p(N)$ is odd, then $f$ is $p$-minimal, in the sense that it has minimal level among its twists by characters of $p$-power conductor.

Looking at elliptic curves whose conductor is highly divisible by 2 or 3, we also found examples of higher ramification indices. These are given in Table 2 below. In this table, we also give examples of ramified cusps for elliptic curves of odd conductor. In all examples we computed, the ramification index seems to be a divisor of $24$. This may be related to the fact that the exponent of the conductor of an elliptic curve at $2$ (resp. $3$) is bounded by $8$ (resp. $5$). It would be interesting to prove this divisibility in general.

\begin{table}[h]
\begin{center}
\begin{tabular}{c|c|c|c}
\textrm{Isogeny class} & $N$ & $d$ & $e_\varphi(\frac1d)$\\
\hline
$405c$  & $3^4 \cdot 5$ & 9 & 3 \\
$768b$  & $2^8 \cdot 3$ & 16 & 8 \\
$891b$  & $3^4 \cdot 11$ & 9 & 3 \\
$1296c$  & $2^4 \cdot 3^4$ & 36 & 6 \\
$1296e$  & $2^4 \cdot 3^4$ & 36 & 12 \\
$20736c$  & $2^8 \cdot 3^4$ & 144  & 24 
\end{tabular}
\end{center}
\caption{Higher ramification indices}
\end{table}

\providecommand{\bysame}{\leavevmode ---\ }
\providecommand{\og}{``}
\providecommand{\fg}{''}
\providecommand{\smfandname}{\&}
\providecommand{\smfedsname}{\'eds.}
\providecommand{\smfedname}{\'ed.}
\providecommand{\smfmastersthesisname}{M\'emoire}
\providecommand{\smfphdthesisname}{Th\`ese}

\end{document}